\documentclass[12pt]{amsart}%
\usepackage{color}
\usepackage{amsmath}
\usepackage{amssymb}
\usepackage{amsfonts}
\usepackage[latin1]{inputenc}
\usepackage{graphicx}%
\providecommand{\U}[1]{\protect\rule{.1in}{.1in}}

\DeclareMathSymbol{\subsetneqq}{\mathbin}{AMSb}{36}

\theoremstyle{plain} \numberwithin{equation}{section}
\newtheorem{theorem}{Theorem}[section]

\newtheorem{lemma}{Lemma}[section]

\begin{document}
\title[Asymptotic for the perturbed heavy ball system]{Asymptotic for the perturbed heavy ball system with vanishing damping term}
\author{}
\maketitle

\begin{center}

\bigskip

Mounir BALTI$^{1,2}$ \& Ramzi MAY$^{3}$

\thinspace\thinspace$^{1}$Institut Pr\'eparatoire aux Etude Scientifiques et Techniques

Universit\'e de Carthage

Bp 51 La Marsa

Tunisia

$^{2}$Facult\'e des Sciences de Tunis

Universit\'e de Tunis Al Manar Tunis

Laboratoire EDP-LR03ES04

Tunisia

E-mail: mounir.balti@gmail.com

$^{3}$ Ramzi May

College of Sciences

Department of Mathematics and Statistics

King Faisal University

P.O. 400 Al Ahsaa 31982, Kingdom of Saudi Arabia

E-mail: rmay@kfu.edu.sa

\end{center}

\noindent\textbf{Abstract:} We investigate the long time behavior of solutions
to the differential equation:%
\begin{equation}
\ddot{x}(t)+\frac{c}{\left(  t+1\right)  ^{\alpha}}\dot{x}(t)+\nabla
\Phi\left(  x(t)\right)  =g(t),~t\geq0, \label{eq}%
\end{equation}
where $c$ is nonnegative constant, $\alpha\in\lbrack0,1[,$ $\Phi$ is a $C^{1}$
convex function on a Hilbert space $\mathcal{H}$ and $g\in L^{1}%
(0,+\infty;\mathcal{H}).$ We obtain sufficient conditions on the source term
$g(t)$ ensuring the weak or the strong convergence of any trajectory $x(t)$ of
(\ref{eq}) as $t\rightarrow+\infty$ to a minimizer of the function $\Phi$ if
one exists.\medskip

\noindent{\textbf{keywords}: Differential equation, asymptotically small
dissipation, asymptotic behavior, energy function, convex function.\medskip}

\noindent{\small \textbf{AMS classification numbers:}} 34G20, 35B40, 35L71, 34D05.

\section{Introduction and main results}

Let $\mathcal{H}$ be a real Hilbert space with inner product and norm
respectively denoted by $\left\langle .,.\right\rangle $ and $\left\Vert
.\right\Vert .$ In this paper, we consider the following second order
equation:
\begin{equation}
\ddot{x}(t)+\gamma\left(  t\right)  \dot{x}(t)+\nabla\Phi\left(  x(t)\right)
=g(t),~t\geq0, \label{aa}%
\end{equation}
where $\gamma(t)=\frac{c}{(1+t)^{\alpha}}$ with $c>0$ and $\alpha\in
\lbrack0,1[,$ $g\in L^{1}(0,+\infty;\mathcal{H})$ and $\Phi:\mathcal{H}%
\rightarrow\mathbb{R}$ is a $C^{1}$ convex function such that its minimizers
subset
\[
\arg\min\Phi:=\{v\in\mathcal{H}:\Phi(v)=\Phi^{\ast}M=\min_{x\in\mathcal{H}%
}\Phi(x)\}
\]
is not empty.

Using classical arguments (see for instance \cite{7}), one can easily prove
that if the function $\nabla\Phi:\mathcal{H}\rightarrow\mathcal{H}$ is
Lipschitz on bounded subset of $\mathcal{H},$ then for any initial data
$(x_{0},x_{1})\in\mathcal{H}\times\mathcal{H},$ the equation (\ref{aa}) has a
unique global solution $x\in W_{loc}^{2,1}(0,+\infty;\mathcal{H})$ satisfying
$(x(0),\dot{x}(0))=(x_{0},x_{1})$. Moreover, the associated energy function
\begin{equation}
W(t)=\frac{1}{2}\left\Vert \dot{x}(t)\right\Vert ^{2}+\Phi(x(t))-\Phi^{\ast
}\label{energy}%
\end{equation}
is nonincreasing and converges to $0$\ as $t\rightarrow+\infty.$ Hence
hereafter, we will assume that $x\in W_{loc}^{2,1}(0,+\infty;\mathcal{H})$ is
a solution to (\ref{aa}) and we will focus our attention on the study of the
asymptotic behavior of $x(t)$ as $t\rightarrow+\infty$ and on the rate of
convergence of the energy function $W.$

Before setting the main results of our present paper, let us first recall some
previous results: In the pioneer paper \cite{2}, Alvarez considered the case
where $\alpha=0$ and $g=0.$ He proved that $x(t)$ converges weakly in
$\mathcal{H}$ as $t\rightarrow+\infty$ to a minimizer of the function $\Phi.$
Moreover, he showed that the convergence is strong if either the function
$\Phi$ is even or the interior of the set $\arg\min\Phi$ is not empty. In
\cite{6}, Haraux and Jendoubi extended the weak convergence result of Alvarez
to the case where the source term $g$\ belongs to the space $L^{1}%
(0,+\infty;\mathcal{H}).$ Recently, Cabot and Frankel \cite{5} studied
(\ref{aa}) where $g=0$ and $\alpha\in]0,1[.$ They proved that every
\underline{bounded} solution $x(t)$ (i.e. $x\in L^{\infty}(0,+\infty
;\mathcal{H})$) converges weakly toward a critical point of $\Phi.$ In a very
recent work \cite{8}, the second author of this paper improved the result of
Cabot and Frankel by getting rid of the superfluous hypothesis on the
boundedness of the solution. Moreover he established that $W(t)=\circ(\frac
{1}{t^{2\bar{\alpha}}})$ as $t\rightarrow+\infty$ for every $\bar{\alpha
}<\alpha.$ In \cite{7}, Jendoubi and May proved that the main convergence
result of Cabot and Frankel remains true if the source term $g$ satisfies the
condition $\int_{0}^{+\infty}(1+t)\left\Vert g(t)\right\Vert dt<\infty$.
Recently, this result was improved in \cite{1}. In fact, we proved that if the
solution $x(t)$ is \underline{bounded} and the function $g$ satisfies the
optimal condition
\begin{equation}
\int_{0}^{+\infty}(1+t)^{\alpha}\left\Vert g(t)\right\Vert dt<\infty,
\label{opt}%
\end{equation}
then $x(t)$ converges weakly to some element of $\arg\min\Phi$ and
$W(t)=\circ(\frac{1}{t^{2\alpha}})$ as $t\rightarrow+\infty.$ One of the main
purpose of this paper is to prove that the sole assumption (\ref{opt})
guarantees the boundedness (and therefore the weak convergence) of the
solution $x(t).$ We notice that, in a very recent work \cite{3}, Attouch,
Chbani, Peypouquet and Redont have considered the equation (\ref{aa}) in the
case $\alpha=1.$ They have proved that if $c>3$ and $\int_{0}^{+\infty
}(1+t)\left\Vert g(t)\right\Vert dt<\infty$ then $x(t)$ converges weakly to
some element of $\arg\min\Phi$ and that $W(t)=O(\frac{1}{t^{2}})$. Moreover,
they have established the strong convergence of $x(t)$ in the case where the
function $\Phi$ is even or the interior of the subset $\arg\min\Phi$ is not
empty. In this paper, we extend their results to the case $\alpha<1.$

Our main first result is the weak convergence of the trajectories of
(\ref{aa}) under the optimal condition (\ref{opt})\ on the source term $g.$

\begin{theorem}
\label{Th1}Assume that $\int_{0}^{+\infty}\left(  t+1\right)  ^{\alpha
}\left\Vert g(t)\right\Vert dt<\infty.$ Then $x(t)$ converges weakly in
$\mathcal{H}$ as $t\rightarrow+\infty$ to some $x^{\ast}\in\arg\min\Phi.$
Moreover the energy function $W$ satisfies the two following properties:%
\begin{equation}
W(t)=O\left(  \frac{1}{t^{2\alpha}}\right)  \label{c1}%
\end{equation}%
\begin{equation}
\int_{0}^{+\infty}(1+t)^{\alpha}W(t)dt<\infty. \label{c2}%
\end{equation}

\end{theorem}

Our second theorem improves the result on the convergence rate of the energy
function $W$\ obtained in \cite{8} in the case where $g=0$ and it will be
useful in the proof of the strong convergence of the solution $x(t)$ when the
convex function $\Phi$ is even.

\begin{theorem}
\label{Th2}Assume that $\int_{0}^{+\infty}(1+t)^{\nu}\left\Vert
g(t)\right\Vert dt<\infty$ where $\nu\in\lbrack\alpha,\frac{1+\alpha}{2}].$
Then
\begin{equation}
W(t)=o(\frac{1}{t^{2\nu}})\text{ as }t\rightarrow+\infty, \label{eq0}%
\end{equation}%
\begin{equation}
\int_{0}^{+\infty}(1+t)^{2\nu-\alpha}\left\Vert \dot{x}(t)\right\Vert
^{2}dt<\infty. \label{eq00}%
\end{equation}

\end{theorem}

The next result shows that, as in the limit case $\alpha=1$ (see \cite[Theorem 3.1]{3}), the strong convergence of $x(t)$ as $t\rightarrow+\infty$ holds if
the interior of $\arg\min\Phi$ is not empty

\begin{theorem}
\label{Th3}Suppose that $\int_{0}^{+\infty}\left(  t+1\right)  ^{\alpha
}\left\Vert g\left(  t\right)  \right\Vert dt<\infty$ and $int\left(  \arg
\min\Phi\right)  \neq\varnothing.$ Then there exists some $x^{\ast}\in\arg
\min\Phi$ such that $x\left(  t\right)  \rightarrow x^{\ast}$ strongly in
$\mathcal{H}$ as $t\rightarrow+\infty.$
\end{theorem}

In the last theorem, we prove, under an assumption on the source term $g$
slightly stronger than the optimal condition (\ref{opt}), the strong
convergence of the solution $x(t)$ when the potential function $\Phi$ is even

\begin{theorem}
\label{Th4}Suppose that $\int_{0}^{+\infty}\left(  t+1\right)  ^{\frac
{\alpha+1}{2}}\left\Vert g\left(  t\right)  \right\Vert dt<\infty$ and $\Phi$
is even (i,e. $\Phi(-x)=\Phi(x),~\forall x\in\mathcal{H}$). Then $x\left(
t\right)  $ converges strongly in $\mathcal{H}$ as $t\rightarrow+\infty$ to
some $x^{\ast}\in\arg\min\Phi.$
\end{theorem}

\section{Proof of Theorem \ref{Th1}}

The proof makes use of a modified version of the method used Attouch, Chbani,
Peypouquet and Redont in \cite{3}. It relies on the study of a suitable
Lyapunov function $\mathcal{E}$ and uses the following two classical lemmas.

\begin{lemma}
[Gronwall-Bellman lemma]Let $f\in L^{1}\left(  \left[  a,b\right]
,\mathbb{R}_{+}\right)  $ ant $c$ a nonnegative constant. Suppose that $w$ is
a continuous function from $\left[  a,b\right]  $ into $\mathbb{R}$ that
satisfies: for all $t\in\left[  a,b\right]  ,$%
\[
\frac{1}{2}w^{2}\left(  t\right)  \leq\frac{1}{2}c^{2}+\int_{a}^{t}f\left(
s\right)  w\left(  s\right)  ds.
\]
Then, for all $t\in\left[  a,b\right]  ,$%
\[
w\left(  t\right)  \leq c+\int_{a}^{t}f\left(  s\right)  ds.
\]

\end{lemma}

The proof of this lemma is easy and similar to the proof of the classical
Gronwall's lemma.
\begin{lemma}
[Opial's lemma \cite{9}]Let $x:[0,+\infty\lbrack\rightarrow\mathcal{H}.$
Assume that there exists a nonempty subset $S$ of $\mathcal{H}$ such that:

\begin{enumerate}
\item[i)] If $t_{n}\rightarrow+\infty$ and $x(t_{n})\rightharpoonup x$ weakly
in $\mathcal{H}$ , then $x\in S.$

\item[ii)] For every $z\in S,$ $\lim_{t\rightarrow+\infty}\left\Vert
x(t)-z\right\Vert $ exists.
\end{enumerate}
\end{lemma}

The proof of Opial's lemma is easy, see for instance \cite{4}.

Let us now start the proof our theorem. We first define on $[0,+\infty\lbrack$
the function
\begin{equation}
h\left(  t\right)  =e^{\Gamma\left(  t\right)  }\int_{t}^{+\infty}%
e^{-\Gamma\left(  s\right)  }ds, \label{q}%
\end{equation}
where $\Gamma\left(  t\right)  =\int_{0}^{t}\gamma\left(  s\right)  ds.$ A
simple calculation yields that $h$ satisfies the differential equation
\begin{equation}
h^{\prime}\left(  t\right)  -\gamma\left(  t\right)  h\left(  t\right)  +1=0.
\label{q1}%
\end{equation}
Moreover, since
\[
\left(  \frac{-1}{\gamma\left(  s\right)  }e^{-\Gamma\left(  s\right)
}\right)  ^{\prime}=\left(  1+\frac{\gamma^{\prime}(s)}{\gamma^{2}(s)}\right)
e^{-\Gamma\left(  s\right)  }\underset{+\infty}{\backsim}e^{-\Gamma\left(
s\right)  }%
\]
then
\begin{equation}
h(t)\underset{+\infty}{\backsim}\frac{1}{\gamma(t)}. \label{q2}%
\end{equation}
Let $x^{\ast}\in\arg\min\Phi$ and define the function%
\begin{align}
\mathcal{E}\left(  t\right)   &  =2\left(  h\left(  t\right)  \right)
^{2}\left(  \Phi\left(  x\left(  t\right)  \right)  -\Phi^{\ast}\right)
+\left\Vert x\left(  t\right)  -x^{\ast}+h\left(  t\right)  \dot{x}\left(
t\right)  \right\Vert ^{2}\nonumber\\
&  -2\int_{0}^{t}h\left(  s\right)  \left\langle g\left(  s\right)  ,x\left(
s\right)  -x^{\ast}+h\left(  s\right)  \dot{x}\left(  s\right)  \right\rangle
ds. \label{b}%
\end{align}
By differentiating, we obtain%
\begin{align*}
\mathcal{E}^{\prime}\left(  t\right)   &  =4h^{\prime}\left(  t\right)
h\left(  t\right)  \left(  \Phi\left(  x\left(  t\right)  \right)  -\Phi
^{\ast}\right)  +2\left(  h\left(  t\right)  \right)  ^{2}\left\langle
\nabla\Phi\left(  x\left(  t\right)  \right)  ,\dot{x}\left(  t\right)
\right\rangle \\
&  +2\left\langle \left(  1+h^{\prime}\left(  t\right)  \right)  \dot
{x}\left(  t\right)  +h\left(  t\right)  \ddot{x}\left(  t\right)  ,x\left(
t\right)  -x^{\ast}+h\left(  t\right)  \dot{x}\left(  t\right)  \right\rangle
\\
&  -2h\left(  t\right)  \left\langle g\left(  t\right)  ,x\left(  t\right)
-x^{\ast}+h\left(  t\right)  \dot{x}\left(  t\right)  \right\rangle .
\end{align*}
Hence by sing (\ref{aa}), we get%
\begin{equation}
\mathcal{E}^{\prime}\left(  t\right)  =4h^{\prime}\left(  t\right)  h\left(
t\right)  \left(  \Phi\left(  x\left(  t\right)  \right)  -\Phi^{\ast}\right)
-2h\left(  t\right)  \left\langle \nabla\Phi\left(  x\left(  t\right)
\right)  ,x\left(  t\right)  -x^{\ast}\right\rangle . \label{b1}%
\end{equation}
Since the function $\Phi$ is convex, we have
\[
\Phi^{\ast}=\Phi\left(  x^{\ast}\right)  \geq\Phi\left(  x\left(  t\right)
\right)  +\left\langle \nabla\Phi\left(  x\left(  t\right)  \right)  ,x^{\ast
}-x\left(  t\right)  \right\rangle .
\]
Inserting this inequality in (\ref{b1}) yields
\[
\mathcal{E}^{\prime}\left(  t\right)  \leq2\left[  2h^{\prime}\left(
t\right)  -1\right]  h\left(  t\right)  \left(  \Phi\left(  x\left(  t\right)
\right)  -\Phi^{\ast}\right)  .
\]
\ From (\ref{q1}) and (\ref{q2}), $2h^{\prime}\left(  t\right)  -1\rightarrow
-1$ as $t\rightarrow+\infty.$ Then there exists $t_{1}\geq0$ such that
\begin{equation}
\mathcal{E}^{\prime}\left(  t\right)  +h\left(  t\right)  \left(  \Phi\left(
x\left(  t\right)  \right)  -\Phi^{\ast}\right)  \leq0,~\forall t\geq t_{1}.
\label{b2}%
\end{equation}
Therefore $\mathcal{E}$ is a decreasing function on $[t_{1},+\infty\lbrack$.
Then for every $t\geq t_{1},$ $\mathcal{E}(t)\leq\mathcal{E}(t_{1}),$ which
implies that
\begin{align}
&  2\left(  h\left(  t\right)  \right)  ^{2}\left(  \Phi\left(  x\left(
t\right)  \right)  -\Phi^{\ast}\right)  +\left\Vert x\left(  t\right)
-x^{\ast}+h\left(  t\right)  \dot{x}\left(  t\right)  \right\Vert
^{2}\nonumber\\
&  \leq C^{2}+2\int_{t_{1}}^{t}h\left(  s\right)  \left\langle g\left(
s\right)  ,x\left(  s\right)  -x^{\ast}+h\left(  s\right)  \dot{x}\left(
s\right)  \right\rangle ds, \label{b3}%
\end{align}
where
\[
C^{2}=2\left(  h\left(  t_{1}\right)  \right)  ^{2}\left(  \Phi\left(
x\left(  t_{1}\right)  \right)  -\Phi^{\ast}\right)  +\left\Vert x\left(
t_{1}\right)  -x^{\ast}+h\left(  t_{1}\right)  \dot{x}\left(  t_{1}\right)
\right\Vert ^{2}.
\]
Using now the Cauchy-Schwarz inequality, we obtain%
\[
\frac{1}{2}\left\Vert x\left(  t\right)  -x^{\ast}+h\left(  t\right)  \dot
{x}\left(  t\right)  \right\Vert ^{2}\leq\frac{C^{2}}{2}+\int_{t_{1}}%
^{t}h\left(  s\right)  \left\Vert g\left(  s\right)  \right\Vert \left\Vert
x\left(  s\right)  -x^{\ast}+h\left(  s\right)  \dot{x}\left(  s\right)
\right\Vert ds.
\]
Hence by applying the Gronwall-Bellman lemma we obtain%

\[
\left\Vert x\left(  t\right)  -x^{\ast}+h\left(  t\right)  \dot{x}\left(
t\right)  \right\Vert \leq C+\int_{t_{1}}^{t}h\left(  s\right)  \left\Vert
g\left(  s\right)  \right\Vert ds,
\]
which implies, thanks to (\ref{q2}), that%
\begin{equation}
M_{1}=\sup_{t\geq0}\left\Vert x\left(  t\right)  -x^{\ast}+h\left(  t\right)
\dot{x}\left(  t\right)  \right\Vert <+\infty\label{bm}%
\end{equation}
Returning to (\ref{b3}), we then infer that%
\begin{align*}
\sup_{t\geq0}\left(  h\left(  t\right)  \right)  ^{2}\left(  \Phi\left(
x\left(  t\right)  \right)  -\Phi^{\ast}\right)   &  \leq C^{2}+2M_{1}%
\int_{t_{1}}^{+\infty}h\left(  s\right)  \left\Vert g\left(  s\right)
\right\Vert ds\\
&  <+\infty.
\end{align*}
Therefore, we deduce from the expression (\ref{b})\ of the function
$\mathcal{E}$, that
\begin{equation}
\sup_{t\geq0}\left\vert \mathcal{E}(t)\right\vert <+\infty\label{a}%
\end{equation}
Hence by integrating the inequality (\ref{b2}) on $\left[  t_{1},t\right]  $
with $t\geq t_{1},$ we infer that%
\begin{equation}
\int_{0}^{+\infty}h\left(  t\right)  \left(  \Phi\left(  x\left(  t\right)
\right)  -\Phi^{\ast}\right)  dt<+\infty. \label{a1}%
\end{equation}
Taking the inner product of (\ref{aa}) with $\dot{x}\left(  t\right)  ,$ we
obtain
\begin{align*}
\left\langle \ddot{x}\left(  t\right)  +\nabla\Phi\left(  x\left(  t\right)
\right)  ,\dot{x}\left(  t\right)  \right\rangle +\gamma\left(  t\right)
\left\Vert \dot{x}\left(  t\right)  \right\Vert ^{2}  &  =\left\langle
g\left(  t\right)  ,\dot{x}\left(  t\right)  \right\rangle \\
&  \leq\left\Vert g\left(  t\right)  \right\Vert \left\Vert \dot{x}\left(
t\right)  \right\Vert \\
&  \leq\sqrt{2}\left\Vert g\left(  t\right)  \right\Vert W(t).
\end{align*}
Multiplying the last inequality by $h^{2}(t)$ and using the fact that
\[
\dot{W}(t)=\left\langle \ddot{x}\left(  t\right)  +\nabla\Phi\left(  x\left(
t\right)  \right)  ,\dot{x}\left(  t\right)  \right\rangle ,
\]
we get after integration by parts on $[0,t],$%
\begin{align*}
&  \left(  h\left(  t\right)  \right)  ^{2}W(t)+\int_{0}^{t}\left(
\gamma\left(  s\right)  \left(  h\left(  s\right)  \right)  ^{2}-\dot
{h}\left(  s\right)  \left(  h\left(  s\right)  \right)  \right)  \left\Vert
\dot{x}\left(  s\right)  \right\Vert ^{2}ds\\
&  \leq h(0)W(0)+\int_{0}^{t}2\dot{h}\left(  s\right)  h\left(  s\right)
\left[  \Phi\left(  x\left(  s\right)  \right)  -\Phi^{\ast}\right]
ds+\sqrt{2}\int_{0}^{t}\left(  h\left(  s\right)  \right)  ^{2}\left\Vert
g\left(  s\right)  \right\Vert \sqrt{W(s)}ds.
\end{align*}
Using now (\ref{q1}), the fact the function $\dot{h}$ is bounded, and
(\ref{a1}), we obtain%
\[
\left(  h\left(  t\right)  \right)  ^{2}W(t)+\int_{0}^{t}h\left(  s\right)
\left\Vert \dot{x}\left(  s\right)  \right\Vert ^{2}ds\leq C+\sqrt{2}\int
_{0}^{t}\left(  h\left(  s\right)  \right)  ^{2}\left\Vert g\left(  s\right)
\right\Vert \sqrt{W(s)}ds,
\]
where $C$ is an absolute constant.

\noindent Hence by applying Gronwall-Bellman lemma with
$\omega=h\sqrt{W}$ and using the fact that
\[
\int_{0}^{\div\infty}h\left(  s\right)  \left\Vert g\left(  s\right)
\right\Vert ds<+\infty,
\]
we deduce that
\begin{equation}
\sup_{t\geq0}h(t)\sqrt{W(t)}<+\infty\label{bm1}%
\end{equation}
(which is equivalent to (\ref{c1})) and therefore
\begin{equation}
\int_{0}^{+\infty}h(s)\left\Vert \dot{x}\left(  s\right)  \right\Vert
^{2}ds<+\infty. \label{a2}%
\end{equation}
Combining (\ref{a1}) and (\ref{a2}), we get (\ref{c2}). Let us now prove the
weak convergence of $x(t)$ as $t\rightarrow+\infty.$ We first notice that
since $W(t)\rightarrow0$ as $t\rightarrow+\infty$ and $\Phi$ is weak lower
semi-continuous (in fact $\Phi$ is continuous and convex), then the first item
i) of Opial's lemma is satisfied with $S=\arg\min\Phi.$ Hence, it remains to
prove that for any $x^{\ast}$ in $\arg\min\Phi,$ the associated function
$z(t):=\frac{1}{2}\left\Vert x(t)-x^{\ast}\right\Vert ^{2}$ converges as
$t\rightarrow+\infty.$ A simple calculation using (\ref{aa}) gives
\[
\ddot{z}\left(  t\right)  +\gamma\left(  t\right)  \dot{z}\left(  t\right)
=\left\Vert \dot{x}\left(  t\right)  \right\Vert ^{2}-\left\langle \nabla
\Phi(x(t)),x\left(  t\right)  -x^{\ast}\right\rangle +\left\langle
g(t),x\left(  t\right)  -x^{\ast}\right\rangle .
\]
Hence by using the monotonicity property of the operator $\nabla\Phi,$ the
fact that $\nabla\Phi(x^{\ast})=0,$ and the Cauchy-Schwarz inequality we get%
\begin{equation}
\ddot{z}\left(  t\right)  +\gamma\left(  t\right)  \dot{z}\left(  t\right)
\leq\left\Vert \dot{x}\left(  t\right)  \right\Vert ^{2}+\left\Vert x\left(
t\right)  -x^{\ast}\right\Vert \left\Vert g\left(  t\right)  \right\Vert
=k(t). \label{A}%
\end{equation}
Combining (\ref{bm}) and (\ref{bm1}), we deduce that
\[
\sup_{t\geq0}\left\Vert x\left(  t\right)  -x^{\ast}\right\Vert <+\infty,
\]
which implies, thanks to (\ref{opt}) and (\ref{a2}), that
\begin{equation}
\int_{0}^{+\infty}\frac{1}{\gamma(t)}k(t)dt<+\infty. \label{B}%
\end{equation}
Multiply now the inequality (\ref{A}) by $e^{\Gamma\left(  t\right)  }$ and
integrate over $\left[  0,t\right]  ,$ we get after simplification the
following inequality
\begin{equation}
\dot{z}\left(  t\right)  \leq K(t):=e^{-\Gamma\left(  t\right)  }%
z(0)+e^{-\Gamma\left(  t\right)  }\int_{0}^{t}e^{\Gamma(s)}k\left(  s\right)
ds. \label{C}%
\end{equation}
By Fubini theorem, we have
\[
\int_{0}^{+\infty}K\left(  t\right)  dt=z(0)\int_{0}^{+\infty}e^{-\Gamma
\left(  t\right)  }dt+\int_{0}^{+\infty}k(s)h(s)ds,
\]
where $h$ is the function defined by (\ref{q}) at the beginning of the proof.
Hence, by using (\ref{q2}) and (\ref{B}) we deduce that the function $K,$ and
therefore the positive part $[\dot{z}]^{+}(t)$ of $\dot{z}(t)$ belongs to the space $L^{1}\left(
0,+\infty\right)  .$ Then the limit of $z\left(  t\right)  $ as $t\rightarrow
+\infty$ exists. This proves the item ii) of the Opial's lemma and completes
the proof of Theorem \ref{Th1}.

\section{Proof of Theorem \ref{Th2}}

By differentiating the energy function $W$ and using the equation (\ref{aa}),
we obtain
\begin{align}
\dot{W}(t) &  =\left\langle \ddot{x}\left(  t\right)  +\nabla\Phi\left(
x\left(  t\right)  \right)  ,\dot{x}\left(  t\right)  \right\rangle
\nonumber\\
&  =-\gamma(t)\left\Vert \dot{x}\left(  t\right)  \right\Vert ^{2}+\langle
g(t),\dot{x}\left(  t\right)  \rangle\nonumber\\
&  \leq-\gamma(t)\left\Vert \dot{x}\left(  t\right)  \right\Vert
^{2}+\left\Vert g(t)\right\Vert \sqrt{2W(t)}\label{Ka}\\
&  \leq\left\Vert g(t)\right\Vert \sqrt{2W(t)}\label{Kb}%
\end{align}
Hence the function $\rho(t):=(1+t)^{2\nu}W(t)$ satisfies the differential
inequality%
\begin{equation}
\dot\rho(t)\leq2\nu(1+t)^{2\nu-1}W(t)+\sqrt{2}(1+t)^{\nu}\left\Vert
g(t)\right\Vert \sqrt{\rho(t)}.\label{e1}%
\end{equation}
Now since $2\nu-1\leq\alpha,$ we have from (\ref{c2}),
\begin{equation}
\int_{0}^{+\infty}(1+t)^{2\nu-1}W(t)dt<\infty.\label{e2}%
\end{equation}
Thus by integrating the differential inequality (\ref{e1}) and applying
Gronwall-Bellman lemma we deduce that $\sup_{t\geq0}\rho(t)<\infty.$ Therefore
(\ref{e1}) and (\ref{e2}) imply that the positive part $[\dot\rho
(t)]^{+}$ of $\dot\rho(t)$ belongs to $L^{1}(0,+\infty)$. Thus $\rho(t)$ converges as $t\rightarrow+\infty$ to some real number
$\lambda$ which in view of (\ref{e2}) must be equal to $0.$ This proves
(\ref{eq0}). Now multiply (\ref{Ka}) by $(1+t)^{2\nu}$ and then integrate on
$[0,t]$ with $t>0,$ we obtain%
\begin{align*}
\int_{0}^{t}(1+s)^{2\nu}\gamma(s)\left\Vert \dot{x}\left(  s\right)
\right\Vert ^{2}ds &  \leq\sup_{s\geq0}\sqrt{2(1+s)^{2\nu}W(s)}\int
_{0}^{+\infty}(1+s)^{\nu}\left\Vert g(s)\right\Vert ds\\
&  +W(0)-(1+t)^{2\nu}W(t)+2\nu\int_{0}^{+\infty}(1+s)^{2\nu-1}W(s)ds,
\end{align*}
which implies (\ref{eq00}) thanks to (\ref{e2}).

\section{Proof of Theorem \ref{Th3}}

We follow the same method used in the proof of \cite[Theorem 3.1]{3}. The assumption $int(\arg\min
\Phi)\neq\varnothing$ implies the existence of $z_{0}\in\mathcal{H}$ and $r>0$
such that for any $v\in\mathcal{H}$ with $\left\Vert v\right\Vert \leq1$ we
have $\nabla(z_{0}+rv)=0$ which implies by the monotocity property of
$\nabla\Phi$ that for any $z\in\mathcal{H}$ we have $\langle\nabla
\Phi(z),z-z_{0}-rv\rangle\geq0.$ Thus by taking the supremum on $v,$ we get%
\[
\langle\nabla\Phi(z),z-z_{0}\rangle\geq r\left\Vert \nabla\Phi\right\Vert .
\]
Hence (\ref{b1}) with $x^{\ast}=z_{0}$ gives
\[
\mathcal{E}^{\prime}\left(  t\right)  +2h(t)r\left\Vert \nabla\Phi
(x(t))\right\Vert \leq4\dot{h}\left(  t\right)  h\left(  t\right)  \left(
\Phi\left(  x\left(  t\right)  \right)  -\Phi^{\ast}\right)  .
\]
Integrating this inequality on $[0,t]$\ and using (\ref{q2}),
(\ref{a}), (\ref{a1}), and the boundedness of $\dot{h},$ we deduce
that
\begin{equation}
\int_{0}^{+\infty}\frac{1}{\gamma\left(  t\right)  }\left\Vert \nabla
\Phi(x(t))\right\Vert dt<+\infty.\label{h1}%
\end{equation}
Setting $\omega(t)=g(t)-\nabla\Phi(x(t)),$ the equation (\ref{aa})\ becomes%

\[
\ddot{x}\left(  t\right)  +\gamma\left(  t\right)  \dot{x}\left(  t\right)
=\omega\left(  t\right)  .
\]
Hence thanks to (\ref{h1}), the following lemma completes the proof of Theorem
\ref{Th3}.

\begin{lemma}
\label{lem2}Let $\omega:[0,+\infty\lbrack\rightarrow\mathcal{H}$ be a
measurable function that satisfies
\[
\int_{a}^{+\infty}\frac{1}{\gamma\left(  t\right)  }\left\Vert \omega\left(
t\right)  \right\Vert dt<+\infty.
\]
If $y\in W_{loc}^{2,1}(0,+\infty;\mathcal{H})$ is a solution of the
differential equation%
\begin{equation}
\ddot{y}\left(  t\right)  +\gamma\left(  t\right)  \dot{y}\left(  t\right)
=\omega\left(  t\right)  , \label{h}%
\end{equation}
then $y\left(  t\right)  $ converge strongly in $\mathcal{H}$ as
$t\rightarrow+\infty.$
\end{lemma}

\begin{proof}
Multiply (\ref{h}) by $e^{\Gamma(t)}$ and integrate on $[0,t],$ we obtain
\[
\dot{y}\left(  t\right)  =e^{-\Gamma(t)}y(0)+e^{-\Gamma(t)}\int_{0}%
^{t}e^{\Gamma(s)}\omega(s)ds.
\]
Hence by using Fubini theorem and (\ref{q}) as in the proof of Theorem
\ref{Th1}, we get%
\begin{align*}
\int_{0}^{+\infty}\left\Vert \dot{y}\left(  t\right)  \right\Vert dt &
\leq\left\Vert \dot{y}(0)\right\Vert \int_{0}^{+\infty}e^{-\Gamma(t)}%
dt+\int_{0}^{+\infty}h(t)\left\Vert \omega(t)\right\Vert dt\\
&  <+\infty\text{ (thanks to (\ref{q2})).}%
\end{align*}
This completes the proof of the lemma.
\end{proof}

\section{Proof of Theorem \ref{Th4}}

Let $T>0.$ We define on $[0,T]$\ the function:%
\[
y\left(  t\right)  =\left\Vert x\left(  t\right)  \right\Vert ^{2}-\left\Vert
x\left(  T\right)  \right\Vert ^{2}-\frac{1}{2}\left\Vert x\left(  t\right)
-x\left(  T\right)  \right\Vert ^{2}.
\]
By a classical calculus using (\ref{aa}) and the fact that $\Phi$ is convex
and even, we obtain%
\begin{align*}
\ddot{y}\left(  t\right)  +\gamma\left(  t\right)  \dot{y}\left(  t\right)
&  =\left\Vert \dot{x}\left(  t\right)  \right\Vert ^{2}+\left\langle
\nabla\Phi\left(  x\left(  t\right)  \right)  ,-x\left(  T\right)
-x(t)\right\rangle +\left\langle g\left(  t\right)  ,x\left(  t\right)
+x\left(  T\right)  \right\rangle \\
&  \leq\left\Vert \dot{x}\left(  t\right)  \right\Vert ^{2}+\Phi
(-x(T))-\Phi(x(t))+M\left\Vert g(t)\right\Vert \\
&  =\left\Vert \dot{x}\left(  t\right)  \right\Vert ^{2}+\Phi(x(T))-\Phi
(x(t))+M\left\Vert g(t)\right\Vert \\
&  \leq\frac{3}{2}\left\Vert \dot{x}\left(  t\right)  \right\Vert
^{2}+W(T)-W(t)+M\left\Vert g(t)\right\Vert ,
\end{align*}
where $M=2\sup_{s\geq0}\left\Vert x(s)\right\Vert $ (Recall that from Theorem
\ref{Th1} we have $x\in L^{\infty}(0,+\infty,\mathcal{H})).$

\noindent Using now the inequality (\ref{Kb}), we get%
\[
\ddot{y}\left(  t\right)  +\gamma\left(  t\right)  \dot{y}\left(  t\right)
\leq\frac{3}{2}\left\Vert \dot{x}\left(  t\right)  \right\Vert ^{2}%
+M\left\Vert g(t)\right\Vert +\sqrt{2}\int_{t}^{T}\left\Vert g(s)\right\Vert
\sqrt{W(s)}ds.
\]
Applying Theorem \ref{Th2} with $\nu=\frac{1+\alpha}{2},$ we deduce the
existence of an absolute constant $C$ such that%
\[
\ddot{y}\left(  t\right)  +\gamma\left(  t\right)  \dot{y}\left(  t\right)
\leq\frac{3}{2}\left\Vert \dot{x}\left(  t\right)  \right\Vert ^{2}%
+M\left\Vert g(t)\right\Vert +C\int_{t}^{+\infty}(1+s)^{-\frac{1+\alpha}{2}%
}\left\Vert g(s)\right\Vert ds:=\omega(t).
\]
Therefore we have
\[
\dot{y}\left(  t\right)  \leq e^{-\Gamma(t)}y(0)+e^{-\Gamma(t)}\int_{0}%
^{t}e^{\Gamma(s)}\omega(s)ds.
\]
Integrating this inequality and using the fact that $y(T)=0,$ we obtain%
\begin{align*}
-y(t) &  =\int_{t}^{T}\dot{y}\left(  s\right)  ds\\
&  \leq\left\vert y(0)\right\vert \int_{t}^{+\infty}e^{-\Gamma(s)}ds+\int
_{t}^{+\infty}e^{-\Gamma(s)}\int_{0}^{s}e^{\Gamma(\tau)}\omega(\tau)d\tau ds\\
&  :=\varpi(t).
\end{align*}
Therefore, for every $t$ in $[0,T],$%
\begin{equation}
\frac{1}{2}\left\Vert x\left(  t\right)  -x\left(  T\right)  \right\Vert
^{2}\leq\left\Vert x\left(  t\right)  \right\Vert ^{2}-\left\Vert x\left(
T\right)  \right\Vert ^{2}+\varpi(t).\label{kh}%
\end{equation}
Since $\Phi$ is even and convex, we have $0\in\arg\min\left(  \Phi\right)  .$
Hence by using the convergence of the function $z\left(  t\right)  =\frac
{1}{2}\left\Vert x\left(  t\right)  -x^{\ast}\right\Vert ^{2}$ proved in
Theorem \ref{Th1} with $x^{\ast}=0$, we infer that the limit of $\left\Vert
x\left(  t\right)  \right\Vert ^{2}$ as $t$ goes to $+\infty$ exists which
implies that
\begin{equation}
\lim_{t,T\rightarrow+\infty}\left\Vert x\left(  t\right)  \right\Vert
^{2}-\left\Vert x\left(  T\right)  \right\Vert ^{2}=0.\label{kh1}%
\end{equation}
On the other hand, in view of Fubini theorem and (\ref{q2}), there exists an
absolute constant $C^{\prime}\geq0$ such that%
\begin{align*}
\int_{0}^{+\infty}e^{-\Gamma(s)}\int_{0}^{s}e^{\Gamma(\tau)}\omega(\tau)d\tau
ds &  =\int_{0}^{+\infty}h(\tau)\omega(\tau)d\tau\\
&  \leq C^{\prime}\int_{0}^{+\infty}(1+\tau)^{\alpha}\omega(\tau)d\tau\\
&  =C^{\prime}\int_{0}^{+\infty}(\frac{3}{2}\left\Vert \dot{x}\left(
\tau\right)  \right\Vert ^{2}+M\left\Vert g(\tau)\right\Vert )(1+\tau
)^{\alpha}d\tau\\
&  +\frac{CC^{\prime}}{1+\alpha}\int_{0}^{+\infty}(1+s)^{\frac{1+\alpha}{2}%
}\left\Vert g(s)\right\Vert ds\\
&  <+\infty\text{ (in view of (\ref{c2})).}%
\end{align*}
Therefore%
\begin{equation}
\lim_{t\rightarrow+\infty}\varpi(t)=0.\label{kh2}%
\end{equation}
Combining (\ref{kh}), (\ref{kh1}) and (\ref{kh2}), we conclude that $x\left(
t\right)  $ satisfies the Cauchy convergence criterion in the Hilbert space
$\mathcal{H}$ as $t\rightarrow+\infty,$ and hence converges strongly in
$\mathcal{H}$ as $t\rightarrow+\infty.$

\end{document}